\documentclass[10pt]{article}
\usepackage{amsmath,amssymb}
\newtheorem{thm}{Theorem}[section]
\newtheorem{co}[thm]{Corollary}
\newtheorem{lem}[thm]{Lemma}
\newtheorem{pr}[thm]{Proposition}
\newcommand{\N}{\mathbb{N}}
\newcommand{\R}{\mathbb{R}}
\newcommand{\Z}{\mathbb{Z}}

\newcommand{\ode}{\mathrm{od}}
\newcommand{\vol}{\mathrm{Vol}\,}

\newcommand{\od}{\mathrm{od}\,}
\newcommand{\tr}{\mathrm{tr}\,}
\newcommand{\diag}{\mathrm{diag}\,}

\newcommand{\openbox}{\leavevmode
  \hbox to.77778em{%
    \hfil\vrule
  \vbox to.675em{\hrule width.6em\vfil\hrule}%
  \vrule\hfil}}
\newcommand{\proofname}{Proof}

\newenvironment{proof}[1][\proofname]{\par\normalfont
  \trivlist\item[\hskip\labelsep\itshape #1:]\ignorespaces
  }{\hspace*{1cm}\hspace*{\fill}\openbox \medskip\endtrivlist}

\title{Some Inequalities Related to the Seysen Measure of a Lattice}%
\date{\today}
\author{G\'erard Maze\\
{\small {\em e-mail:\/} gmaze@math.uzh.ch \vspace{-1mm} }\\
{\small Mathematics Institute\vspace{-1mm}}\\
{\small University of Z\"urich\vspace{-1mm}}\\
{\small Winterthurerstr 190, CH-8057 Z\"urich, Switzerland }
\vspace{3mm} }

\begin{document}\maketitle
\thispagestyle{empty}
\begin{abstract}
  Given a lattice $L$, a basis $B$ of $L$ together with its dual
  $B^*$, the orthogonality measure $S(B)=\sum_i ||b_i||^2 ||b_i^*||^2$
  of $B$ was introduced by M. Seysen \cite{seysen1} in 1993. This
  measure (the Seysen measure in the sequel, also known as the {\it
    Seysen metric} \cite{zhang}) is at the heart of the Seysen lattice
  reduction algorithm and is linked with different geometrical
  properties of the basis \cite{ling1, ling2, seysen2, zhang}. In this
  paper, we derive different expressions for this measure as well as
  new inequalities related to the Frobenius norm and the condition
  number of a matrix.
\end{abstract}
\vspace{3mm}


\vspace{3mm}

%%%%%%%%%%%%%%%%%%%%%%%%%%%%%%%%%%%%%%%%%%%%%%%%%%%%%%%%%%%%%%%%%%%
\section{Introduction, Notations and Previous Results}

An $n$-dimensional (real) lattice $L$ is defined as a subset of
$\R^m$, $n\leq m$, generated by $B = [b_1 | \ldots | b_n]^t$, where
the $b_i$ are $n$ linearly independent vectors over $\R$ in $\R^m$, as
\[
L=\left\{ \sum_{i=1}^n a_i b_i \, | \, a_i \in \Z\right\}.
\]

In this paper, the rows of the matrix $B$ span the lattice. Any other
matrix $B'= U B$, where $U \in GL_n(\Z)$, generates the same
lattice. The volume $\vol L$ of $L$ is the well defined real number
$(\det BB^t)^{1/2}$. The dual lattice of $L$ is defined by the basis
$B^*= \left( B^{+} \right)^t$, where $B^+$ is the Moore-Penrose
inverse, or pseudo-inverse, of $B$. If $B^* = [b^*_1 | \ldots |
  b^*_n]^t$, then since $B B^+ = I_n$, we have $\langle b_i,b^*_j
\rangle = \delta_{i,j} $. Lattice reduction theory deals with the
problem of identifying and computing bases of a given lattice whose
vectors are {\it short} and {\it almost orthogonal}. There are several
concepts of reduced bases, such as the concepts of Minkovsky reduced,
LLL reduced \cite{lenstra} and Korkin-Zolotarev reduced basis
\cite{korkine}. In 1990, Hastad and Lagarias \cite{hastad} proved that
in all lattices of full rank (i.e., when $n=m$), there exists a basis
$B$ such that both $B$ and $B^*$ consist in relatively short vector,
i.e., $\max_i ||b_i||\cdot||b^*_i|| \leq \exp(O(n^{1/3}))$. In 1993,
Seysen \cite{seysen1} improved this upper bound to $\exp(O(\ln^2(n)))$
and suggested to use the expression $S(B):=\sum_i ||b_i||^2
||b_i^*||^2$. This definition also allowed him to define a new concept
of reduction: a basis $B$ of $L$ is Seysen reduced if $S(B)$ is
minimal among all bases of $L$ (see also \cite{lamacchia} for a study
of this reduction method). A relation between the orthogonality defect
\cite{kaltofen,zhang}
\[
\od(B) := 1-\frac{\det B B^t}{\prod_{i=1}^n ||b_i||^2} \in [0,1]
\]
and the Seysen measure $S(B)$ is given in \cite{zhang} where the
following bounds can be found:
\begin{eqnarray}
n \, \leq \,S(B)  & \leq & \frac{n}{1-\ode(B)}, \label{zhang1}\\
0 \, \leq \,\ode(B) & \leq & 1- \frac{1}{\left(S(B)-n+1
  \right)^{n-1}}. \label{zhang2}  
\end{eqnarray}
Clearly, the smaller the Seysen measure is, the closer to orthogonal
the basis is, showing that the Seysen measure describes the quality of
the angle behavior of the vectors in a basis. The length of the
different vectors are nevertheless not part of the direct information
given by the measure, but Inequality \ref{zhang2} gives
\[
\prod_{i=1}^n ||b_i|| \leq \left(S(B)-n+1 \right)^{\frac{n-1}{2}}
\cdot \vol L 
\]
which in turn provides the inequality
\begin{eqnarray}\label{size_bound}
\min_i ||b_i|| \leq \left(S(B)-n+1 \right)^{(n-1)/2n} \left( \vol
L \right)^{1/n}.
\end{eqnarray}
Note that such a type of inequality appears in the context of lattice
reduction as
\[
\begin{array}{rl}
\min_i ||b_i|| \leq \sqrt{n} \left( \vol
L \right)^{1/n} & \mbox{for Korkin Zolotarev and Minkovsky reduced
  bases}\\
\min_i ||b_i|| \leq (4/3)^{(n-1)/4} \left( \vol
L \right)^{1/n} & \mbox{for LLL reduced bases.}\\
\end{array}
\] 

In this paper, we start by revisiting Seysen's bound
$\exp(O(\ln(n)^2))$ by computing the hidden constant in Landau's
notation. Then we present new expressions for the Seysen measure,
connecting the measure with the condition number and the Frobenius
norm of a matrix and allowing us to improve some of the existing
bounds. We will from now on suppose that $m=n$, since Equality
\ref{sin} below shows that the Seysen measure is invariant under
isometric embeddings.

\section{Explicit Constant in Seysen's Bound}

We show in this section that the hidden constant in Seysen's bound
$\exp(O(\ln(n)^2))$ can be upper bounded by $1+\frac{2}{\ln 2}$. The
proof is not new, but revisits some details in the original proof of
Seysen \cite[Theorem 7]{seysen1} by using explicit bounds given in
\cite[Proposition 4.2]{lenstra}. Let us define the two main
ingredients of the proof. First, if $N(n,\R)$ and $N(n,\Z)$ are the
group of lower triangular unipotent $n \times n$ matrices over $\R$
and $\Z$ respectively (i.e. matrices with 1 in the diagonal), then
following \cite{hastad} and \cite{seysen1}, and if $||X||_{\infty} =
\max_{i,j} |X_{ij}|$, we define $S(n)$ for all $n \in \N$ by
\[
S(n) = \sup_{A\in N(n,\R)} \left( \inf_{T\in N(n,\Z)}
\max(||TA||_{\infty},||(TA)^{-1}||_{\infty}) \right).
\]
In \cite{seysen1}, the author proves that $S(2n) \leq S(n)\cdot
\max(1,n/2)$, and concludes that $S(n) = \exp(O((\ln n)^2))$. We would
like to point out that the latter is not true in general, unless some
other property of the function $S$ is invoked. Indeed, an arbitrary
map $s$ defined on the set of odd integers, e.g. $s(2n+1)=\exp(2n+1)$,
and extended to $\N$ with the rule $s(2n) = n/2 \cdot s(n)$ satisfies
the condition $s(2n) \leq s(n)\cdot \max(1,n/2)$ but we have $s(n)
\neq \exp(O((\ln n)^2))$ in general. This point seems to have been
overlooked in \cite{seysen1}. However, in our case, we have the
following in addition.
\begin{lem}
$\forall n\leq m \in \N, S(n) \leq S(m)$
\end{lem}
\begin{proof}
It is not difficult to see that for all $A \in N(n,\R)$, there
exists a matrix $T_A\in N(n,\Z)$ such that
\[
 \inf_{T\in N(n,\Z)}
\max(||TA||_{\infty},||(TA)^{-1}||_{\infty}) =
\max(||T_AA||_{\infty},||(T_AA)^{-1}||_{\infty}). 
\]
See the Remark following Definition 4 of \cite{seysen1} for the
details. As a consequence, in order to prove the lemma, it is
sufficient to show that
\begin{equation}\label{sup}
\sup_{A\in N(n,\R)} \max(||T_AA||_{\infty},||(T_AA)^{-1}||_{\infty})
\leq \sup_{A'\in N(n+1,\R)}
\max(||T_{A'}A'||_{\infty},||(T_{A'}A')^{-1}||_{\infty}).
\end{equation}
Let us consider the map $i$ from $N(n,\R)$ to $N(n+1,\R)$ defined by
mapping a matrix $A$ to the block matrix $\diag (1,A)$. The map $i$ is
a group homomorphism and thus
$i(A)^{-1}=i(A^{-1})=\diag(1,A^{-1})$. We claim that for all $A \in
N(n,\R)$ and all $T\in N(n,\Z)$, we have
\begin{equation}\label{max}
\max(||i(TA)||_{\infty},||i(TA)^{-1}||_{\infty}) =
\max(||TA||_{\infty},||(TA)^{-1}||_{\infty}).
\end{equation}
First, if $\max(||i(TA)||_{\infty},||i(TA)^{-1}||_{\infty})=1 $, then
the above equality is straightforward, due to the definition of
$||.||_{\infty}$. Let us then consider the case where the maximum is
not 1.  Notice that since $||X||_{\infty} \geq 1$ is true for all
matrix $X$ in $N(m,\R)$, we have that
$\max(||X||_{\infty},||X^{-1}||_{\infty}) \geq 1$ and so
$\max(||i(TA)||_{\infty},||i(TA)^{-1}||_{\infty})>1 $. As a
consequence the maximum in $\max(||i(TA)||_{\infty},
||i(TA)^{-1}||_{\infty})$ is achieved by one of the entries of $i(TA)$
or $i(TA)^{-1}$, and this entry cannot be the one in the upper left
corner. The maximum is then the same for both sides of
(\ref{max}). This proves the above claim. Now, since
\[
\sup_{A'\in N(n+1,\R)}
\max(||T_{A'}A'||_{\infty},||(T_{A'}A')^{-1}||_{\infty}) \geq
\max(||i(TA)||_{\infty},||i(TA)^{-1}||_{\infty}) =
\max(||TA||_{\infty},||(TA)^{-1}||_{\infty}),
\]
is true for all $A\in N(n,\R)$, taking the supremum on the left
hand side, we see that Inequality \ref{sup} is correct.
\end{proof}

This lemma makes the following inequalities valid:
\begin{eqnarray*}
S(n) = S(2^{\log_2n}) \leq S(2^{\lceil \log_2n \rceil}) \leq 2^{\lceil
  \log_2n \rceil-2} \cdot 2^{ \lceil \log_2n \rceil-3} \cdot \ldots
\cdot 2 \cdot 1 \leq \exp \left( \frac{(\ln n)^2}{2 \ln 2}\right).
\end{eqnarray*}
The second ingredient we need is related to the Korkin-Zolotarev
reduced bases of a lattice $L$. Such bases are well known, see
e.g. \cite{lenstra}, and one of their properties is the following: if
$B$ is a Korkin-Zolotarev reduced basis of $L$, and if $B=HK$, where
$H=(h_{ij})$ is a lower triangular matrix and $K$ is an orthogonal
matrix, then for all $1\leq i\leq j \leq n$, we have
\[
h_{jj}^2 > h_{ii}^2 (j-i+1)^{-1-\ln(j-i+1)}.
\]
This is a direct consequence of \cite[Proposition 4.2]{lenstra} and
the fact that the concept of Korkin-Zolotarev reduction is
recursive. See \cite{seysen1} for the details. In \cite{seysen1}, the
author concludes that $\frac{h_{ii}^2}{h_{jj}^2} = \exp(O((\ln n)^2))$
but we have the more precise statement that
\[
\frac{h_{ii}^2}{h_{jj}^2} \leq \exp((\ln (j-i+1))^2+\ln (j-i+1)) \leq
\exp((\ln n)^2+\ln n).
\]
Let us now revisit the proof of \cite[Theorem 7]{seysen1} by making
use of the previous inequalities. This theorem states that for every
lattice $L$ there is a basis $\tilde{B} = [\tilde{b_1} | \ldots |
  \tilde{b_n}]^t$ with reciprocal basis $\tilde{B}^* = [\tilde{b_1}^*
  | \ldots | \tilde{b_n}^*]^t$ which satisfies
\[
||\tilde{b_i}|| \cdot ||\tilde{b_i}^*|| \leq \exp(c_2 (\ln n)^2)
\] 
for all $i$ and for a fixed $c_2$, independent of $n$. We explicit
now an upper bound for the constant $c_2$. Given a lattice $L$ and a
Korkin-Zolotarev reduced basis $B=HK$ as above, the proof of
\cite[Theorem 7]{seysen1} shows that there exists a
basis $\tilde{B}$, constructed from $B$, such that
\[
||\tilde{b_i}||^2 \cdot ||\tilde{b_i}^*||^2 \leq n^2 \cdot \max_{k\geq
  j}\left\{ \frac{h_{jj}^2}{h_{kk}^2} \right\} \cdot S(n)^4
\]
Making use of the previous inequalities, we can write
\[
||\tilde{b_i}||^2 \cdot ||\tilde{b_i}^*||^2 \leq n^2 \cdot\exp((\ln
n)^2+\ln n) \cdot \exp \left( \frac{4(\ln n)^2}{2 \ln 2}\right) = \exp
\left( \left(\frac{2}{\ln 2}+1\right) (\ln n)^2 + 3 \ln n\right).
\]
which shows that $c_2 < \frac{1}{\ln 2}+\frac{1}{2} + \frac{3}{2\ln n}
< \frac{1}{\ln 2}+\frac{1}{2} + \frac{3}{2\ln 2} = \frac{5}{2\ln
  2}+\frac{1}{2}$ and gives the following proposition:
\begin{pr}\label{explicit}
For every lattice $L$ there is a basis $B$ which satisfies
\[
S(B) \leq \exp \left( \left(\frac{2}{\ln 2}+1\right) (\ln n)^2 + 4 \ln
n \right).
\]
\end{pr}

%%%%%%%%%%%%%%%%%%%%%%%%%%%%%%%%%%%%%%%%%%%%%%%%%%%%%%%%%%%%%%%%%%%
\section{Explicit Expression for the Seysen Measure}

In this section, we present different expressions for the Seysen
measure. First, let us recall the following known expression for the
measure. Given a basis $B$ of $L$, by definition of $B^*$, for all $0
\leq j \leq n$, the vector $b_j^*$ is orthogonal to $L_j$, where $L_j$
is the sublattice of $L$ generated by all the vectors of $B$ except
$b_j$. If $\beta_j$ is the angle between $b_j$ and $b_j^*$ and
$\alpha_j$ is the angle between $b_j$ and $L_j$, we have $\cos^2
\beta_i=\sin^2 \alpha_i$ and
\begin{equation}\label{sin}
  S(B) = \sum_i ||b_i||^2||b_i^*||^2 = \sum_i \frac{\langle b_i,b^*_i
  \rangle^2 }{\cos^2 \beta_i} = \sum_i \frac{1}{\sin^2 \alpha_i}.
\end{equation}
This has already been used in \cite{lamacchia,zhang}. We introduce now
the following new representation, which can be used to define the
Seysen measure without any references to the dual basis:
\begin{pr}
For every lattice $L$, if $B = [b_1 | \ldots | b_n]^t$ is a basis of
$L$ with $B= D \cdot V$ where $D=\diag(||b_1||,\ldots,||b_n||)$, then
\[
S(B) = ||V^{-1}||^2
\]
where $||.||$ is the Frobenius norm, i.e., $||X||=\sqrt{\sum_{i,j} |x_{ij}|^2}$.
\end{pr}
\begin{proof}
Let $M=BB^t$. Using $||X||^2 = \tr(XX^t)$ and $\tr(ABC)=\tr(CAB)$, we
have
\[
||V^{-1}||^2 = \tr(V^{-1}(V^{-1})^{t}) = \tr(D^2M^{-1}) = \sum_i
||b_i||^2 \cdot \left( M^{-1} \right)_{i,i}.
\]
Since $M^{-1}=\frac{1}{\det M} \mbox{comat}(M)$, where
$\mbox{comat}(M)$ is the comatrix of $M$, we have 
\[
\left( M^{-1}
\right)_{i,i} = \frac{1}{\det M} \mbox{comat }(M)_{i,i} =  \frac{\det
  M^{i,i}}{\det M}
\]
where $ M^{i,i}$ is the square matrix obtained from $M$ by deleting
the $i$-th row and the $i$-th column of $M$. So if $B^i$ is the matrix
obtained by deleting the $i$-th row of $B$, we have
\[
\det M^{i,i} = \det B^i (B^i)^t = \left( \vol L_i \right)^2
\]
which gives
\[
\frac{\det M^{i,i}}{\det M} = \frac{\left( \vol L_i \right)^2}{\left(
  \vol L \right)^2} = \frac{\left( \vol L_i \right)^2}{\left(
  ||b_i||\cdot \vol L_i \cdot \sin \alpha_i \right)^2} =
\frac{1}{||b_i||^2 \sin^2 \alpha_i}.
\]
Finally,
\[
||V^{-1}||^2 = \sum_i ||b_i||^2 \cdot \left( M^{-1} \right)_{i,i} =
 \sum_i ||b_i||^2 \cdot  \frac{1}{||b_i||^2 \sin^2 \alpha_i} = S(B).
\]
\end{proof}
Another way of looking at the previous result is with the help of the
(Frobenius) condition number of an invertible matrix $X$ which is
defined as $\kappa(X)=||X|| \cdot ||X^{-1}||$.
\begin{co}
With the above notation, we have $S(B)=\frac{\kappa(V)^2}{n}$.
\end{co}
By defining the matrix $U$ as $U=VV^t$, then $BB^t=DUD$, where $D$ is
as above, and if $\theta_{ij}$ is the angle between $b_i$ and $b_j$,
then $U=(\cos \theta_{ij})_{ij}$. The matrix $U$ is a symmetric
positive definite matrix, and the eigenvalues
$\lambda_1,\ldots,\lambda_n$ of $U$ are real positive.
\begin{co}
With the above notation, we have $S(B)= \tr (U^{-1}) = \sum_i
\frac{1}{\lambda_i}$.
\end{co}
From the equality $BB^t=DUD$, we have  $\left( \vol L \right)^2 = \det
U \cdot \prod_i ||b_i||^2$ which in turn leads to
\begin{equation}\label{eq}
 \prod_i ||b_i|| =  \left( \det U \right)^{-1/2} \cdot \vol L
 = \left( \prod_i \frac{1}{\lambda_i} \right)^{1/2} \cdot \vol L.
\end{equation}
The arithmetic-geometric mean inequality applied to the
$\lambda_i$'s, $\left( \prod_i 1/\lambda_i \right)^{1/n} \leq
\frac{1}{n} \sum_i 1/\lambda_i$, immediately gives the inequality
\[
 \prod_i ||b_i|| \leq \left( \frac{1}{n} \sum_i \frac{1}{\lambda_i}
 \right)^{\frac{n}{2}} \cdot \vol L = \left( \frac{S(B)}{n}
 \right)^{\frac{n}{2}} \cdot \vol L.
\] 

\noindent However, we also have the equality $\sum_i \lambda_i = \tr U
= n$, which affords a slightly better upper bound for the geometric
mean. Indeed, the harmonic-geometric-arithmetic mean inequalities
applied to the $1/\lambda_i$'s imply that if $g=\left( \prod_i
1/\lambda_i \right)^{1/n}$, $h=\left( \frac{1}{n} \sum_i
\lambda_i\right)^{-1} = 1$ and $a = \frac{1}{n} \sum_i
\frac{1}{\lambda_i} = \frac{S(B)}{n}$, then we have $h\leq g \leq a$,
but we also have the following result, which is \cite[Corollary
  3.1]{maze}.
\begin{lem}
With the above notations, if $\alpha=1/n$, we have
\[
g \leq \left(
\frac{a-h(1-2\alpha)-\sqrt{(a-h)(a-h(1-2\alpha)^2)}}{2\alpha}\right)^{\alpha}
\left(
\frac{a+h(1-2\alpha)+\sqrt{(a-h)(a-h(1-2\alpha)^2)}}{2(1-\alpha)}\right)^{1-\alpha}.
\]
\end{lem}
This leads to the following inequality:

\begin{pr}
With the above notation, we have
\begin{equation}\label{eq2}
 \prod_i ||b_i|| \leq e^{1/2} \cdot \left(\frac{S(B)+1}{n}
 \right)^{\frac{n-1}{2}} \cdot \vol L.
\end{equation}
\end{pr}

\begin{proof}
Since $(1-2/n)^2\leq 1$, we have
\[
(a-h)^2 \leq (a-h)(a-h(1-2/n)^2) \leq (a-h(1-2/n)^2)^2 
\]
and thus the upper bound of the previous Lemma  gives
\begin{eqnarray*}
g & \leq & \left(
\frac{a-h(1-2/n)-(a-h)}{2/n}\right)^{1/n}
\left(
\frac{a+h(1-2/n)+(a-h(1-2/n)^2)}{2(1-1/n)}\right)^{1-1/n}.
\end{eqnarray*}
After suitable simplification, we obtain
\begin{eqnarray*}
g & \leq & a \cdot \left(\frac{h}{a}\right)^{1/n} \cdot
\left(1+\frac{h}{a}\cdot\left(1-\frac{2}{n}\right) \cdot \frac{1}{n}
\right)^{1-1/n} \cdot \left(1+\frac{1}{n-1}\right)^{1-1/n}.
\end{eqnarray*}
Since $(1+\frac{1}{n-1})^{n-1}<e$, taking the $n$-th power of both
sides of the previous inequality gives
\[
\prod_i 1/\lambda_i < e \cdot \left(\frac{S(B)+1-\frac{2}{n}}{n}
\right)^{n-1} <  e \cdot \left(\frac{S(B)+1}{n}
\right)^{n-1} .
\]
The result follows by applying the previous inequality to Equation
(\ref{eq}).
\end{proof}
This is an improvement by a factor of roughly $n^{n/2}$ of the bound
given by (\ref{size_bound}), and can be used to strengthen the bound
of the orthogonality defect (\ref{zhang1}):
\begin{co}
With the above notations, we have
\[
\ode(B) \leq  1- \frac{1}{e}\left(\frac{n}{S(B)+1}\right)^{n-1}
\]
\end{co}
Combining the previous proposition with the explicit bound of
Proposition \ref{explicit}, we have the following proposition:
\begin{pr}
For every lattice $L$, if $B=[b_1|\ldots|b_n]^t$ is a Seysen
reduced basis, then
\begin{equation*}
\min_i ||b_i||
 \leq \exp \left( \left(\frac{1}{\ln 2}+\frac{1}{2}\right) (\ln n)^2
+ O(\ln n) \right)  \cdot \left( \vol L \right)^{1/n}.
\end{equation*}
\end{pr}

\section{Conclusion}

In this article, we gave an explicit upper bound for the constant
hidden inside Landau's notation of the original bound of the Seysen
measure \cite{seysen1}. We also developed the connection between the
Seysen measure and standard linear algebra concepts such as the
Frobenius norm and the condition number of a matrix. This allowed us
to improve known upper bounds for the Seysen measure and the
orthogonality defect.

%%%%%%%%%%%%%%%%%%%%%%%%%%%%%%%%%%%%%%%%%%%%%%%%%%%%%%%%%%%%%%%%%%%

\end{document}